\documentclass[11pt, a4paper]{article} 
\usepackage{amsmath}
\usepackage{amssymb}
\usepackage{amsthm}
\usepackage{mathrsfs}
\usepackage{graphicx}
\usepackage{hyperref}
\hypersetup{
    colorlinks=true,  
    linkcolor=blue,   
    citecolor=blue
}
\usepackage{xcolor}

\newtheorem{theorem}{Theorem}[section]

\newtheorem{lemma}[theorem]{Lemma}
\newtheorem{proposition}[theorem]{Proposition}
\newtheorem{definition}{Definition}

\newcommand{\C}{\mathbb{C}}

\newcommand{\rank}{\text{rank}}
\newcommand{\detp}{{\det}_+}

\begin{document}
\title{Matrix completion and semidefinite matrices}

\author{Olaf Dreyer\thanks{OD Consulting, Frankfurt am Main, email: olaf.dreyer@gmail.com}}

\date{\today}

\maketitle

\begin{abstract} Positive semidefinite Hermitian matrices that are not fully specified can be completed provided their underlying graph is chordal. If the matrix is positive definite the completion can be uniquely characterized as the matrix that maximizes the determinant, or as the matrix whose inverse has zeroes in those places that were undetermined in the original matrix. This paper extends these uniqueness results to the case of semidefinite matrices. Because the determinant vanishes for singular matrices, and because the inverse does not exist, we introduce a generalized determinant and use generalized inverses to formulate equivalent characterizations in the semidefinite case. For a class of matrices that are singular but of maximal rank unique characterizations can be given, just as in the positive definite case. 
\end{abstract}

\emph{AMS classification:} 15A10; 15A15; 15A18; 15A83; 15B48
 \\
 
\emph{Keywords:} positive definite; positive semidefinite; matrix completions 

\section{Introduction}
The question of which partial Hermitian matrices $H$ can be completed to give a fully specified positive definite Hermitian matrix was solved in \cite{grone}. The solution provided in \cite{grone} was constructive but only gave one element of the completion with each step. While giving the correct solution, the practical implementation of this procedure was rather tedious. In \cite{smith} a faster procedure was proposed that delivered the same matrix in far fewer steps by focusing on whole blocks of the matrix at once. 

The starting point of both procedures is the graph $\gamma$ that is constructed from the partial matrix. The vertices of $\gamma$ are given by the rows (or columns) $i$, $i=1,\ldots,n$, of the matrix $H$. Two vertices $i$ and $j$ of $\gamma$ are connected by an edge $e$ if the element 
	\begin{equation}
		h_{ij}
	\end{equation}
of the matrix is given. There is now one condition that the graph $\gamma$ has to satisfy for the procedure to work: The graph has to be \emph{chordal}. A graph is chordal if every loop of four or more elements has a chord, i.e. an edge that connects two nonconsecutive elements of the loop (for an introduction to graph theoretic notions see \cite{golumbic}). Given the graph $\gamma$ we can now proceed by constructing another graph $\Gamma$. The vertices of $\Gamma$ are given by the cliques of $\gamma$ (a clique is a maximal set of vertices that are all connected to each other; again see \cite{golumbic}) and two vertices of $\Gamma$ are connected if the two corresponding cliques have a nonempty intersection (see figure \ref{fig.allgraphs}).

\begin{figure}[hbt]
  \begin{center}
  	\includegraphics[width=12cm]{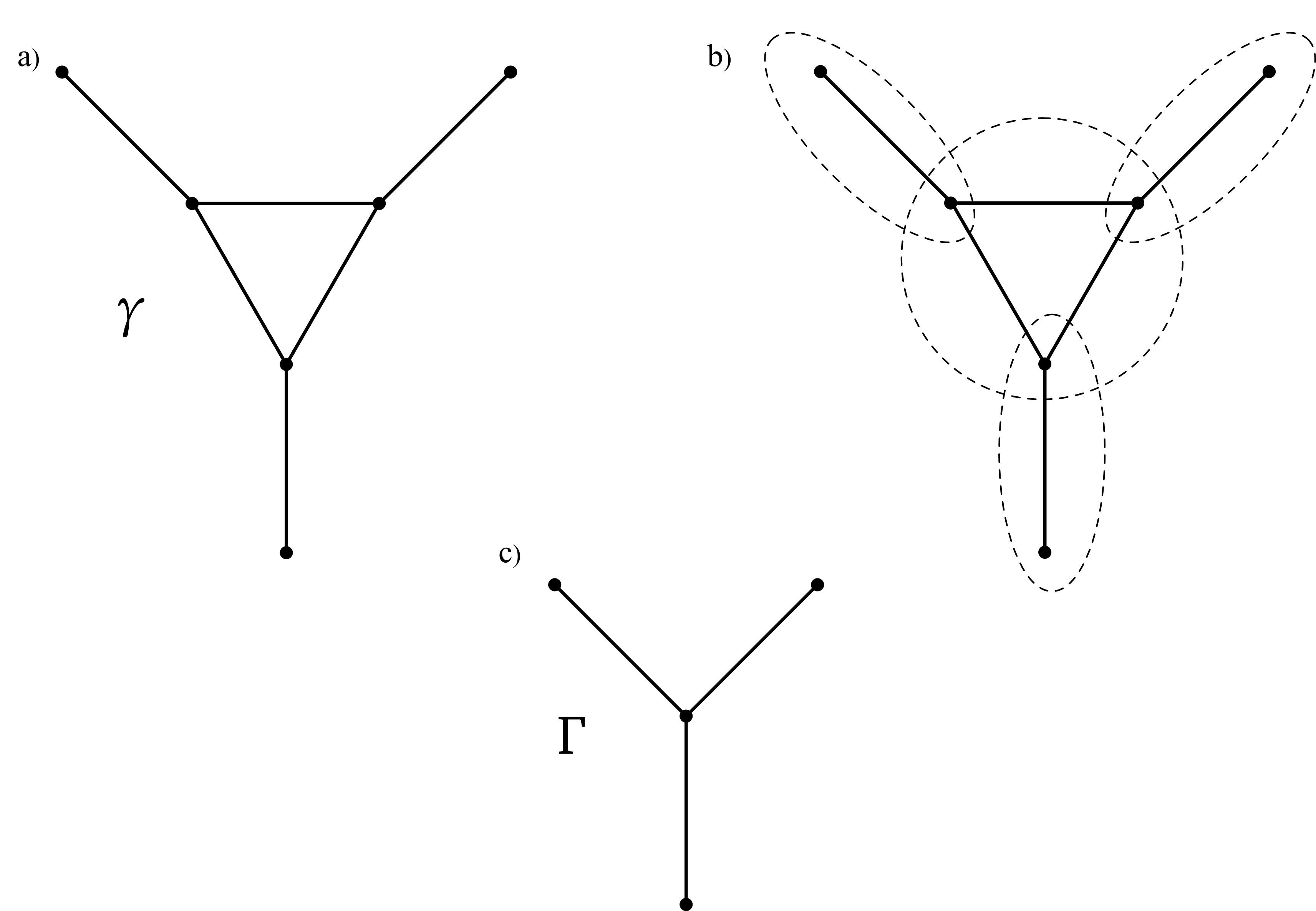}
  \end{center}
  \caption{a) Given a partial matrix $H$ we can construct a graph $\gamma$ that has an edge for all the specified elements of the matrix. The vertices of the graph are the rows (or columns) of the matrix. For the completion procedure to work the graph $\gamma$ needs to be chordal. The graph shown here is chordal because the central loop only has three elements and thus does not require a chord. b) The dashed lines show the cliques of $\gamma$. c) The cliques are the vertices of the new graph $\Gamma$. Edges in $\Gamma$ connect cliques with a nonempty intersection.}\label{fig.allgraphs}
\end{figure}

The contribution of \cite{smith} is to show that the completion procedure can be reduced to edges of the graph $\Gamma$. Any edge in $\Gamma$ is given by two cliques $\alpha$ and $\beta$ that have a nonzero intersection $\alpha\cap\beta$. Let 
\begin{equation}
	H[\alpha] = \begin{pmatrix}
		A & B \\
		B^\star & C 
	\end{pmatrix}
\end{equation}
be the submatrix of $H$ with rows and column indices in $\alpha$ and let 
\begin{equation}
	H[\beta] = \begin{pmatrix}
		C & D \\
		D^\star & E 
	\end{pmatrix}
\end{equation}
be the submatrix of $H$ with rows and column indices in $\beta$. We then have
\begin{equation}
	H[\alpha\cap\beta] = C.
\end{equation}
We can combine these two matrices in one large matrix with row and column indices in $\alpha\cup\beta$,
\begin{equation}
	H[\alpha\cup\beta] = \begin{pmatrix}
		A & B & X \\
		B^\star & C & D \\
		X^\star & D^\star & E
	\end{pmatrix},
\end{equation}
where the matrix $X$ remains to be specified. It was now shown in \cite[Theorem 3.2.]{smith} that a partial positive semidefinite matrix $H[\alpha\cup\beta]$ can be completed to a positive semidefinite matrix by choosing
\begin{equation}\label{eqn.defx}
	X = B C^+ D,
\end{equation}
where $C^+$ denotes the Moore-Penrose inverse of $C$ (see \cite{roger} and \cite{matrixanalysis} for details on the Moore-Penrose inverse). Note that the matrix does not need to be positive definite for the construction to work. It it sufficient for the matrix to be positive \emph{semi}definite. 

If the matrix is positive definite the choice of $X$ in equation (\ref{eqn.defx}) can be uniquely characterized in two ways. First, it is the choice that maximizes the determinant of the completed matrix. It is also the choice for which the inverse of $H[\alpha\cup\beta]$ has zeroes in those places where $X$ sits in $H[\alpha\cup\beta]$.

We see that the treatment of positive definite and positive semidefinite matrices differs. For positive definite matrices the completion in equation (\ref{eqn.defx}) has properties that uniquely determine it. For the semidefinite case we are just left with the existence result. Since the determinant of a singular matrix vanishes and since the inverse of a singular matrix does not exist it is not clear if we can hope to improve this situation. 

In this paper we now make two contributions. The first contribution is a new proof that the $X$ from equation (\ref{eqn.defx}) gives a positive semidefinite completion of $H[\alpha\cup\beta]$. We then extend the uniqueness results to the semidefinite case. To do this we need to introduce a generalized determinant that gives the determinant of the nonsingular part of the matrix. This determinant lacks many of the nice properties that the usual determinant has. If we restrict our attention to a special class of matrices, though, we can prove three results that hold for positive definite matrices: Fischer's inequality, the Schur determinant lemma, and Banachiewicz's form of the inverse. These three results will then be used to uniquely characterize the completion in equation (\ref{eqn.defx}) for semidefinite matrices. 

We start by introducing the class of matrices for which our results are valid. 

\section{Partitioned matrices of maximal rank}\label{sec.partition}
Let the Hermitian matrix $H\in M_n(\C)$ be partitioned as follows:
\begin{equation}
	H = \begin{pmatrix}\label{eqn.defh}
		A & B \\
		B^\star & C
	\end{pmatrix},
\end{equation}
with $A\in M_k(\C)$, $C\in M_l(\C)$, $1\le k,l\le n-1$, $k+l=n$. Let $V = K\oplus L$ be the decomposition of $V\simeq \C^n$ in accordance with the partition of $H$ in equation (\ref{eqn.defh}) so that $A$ is a map from $K$ to $K$, and $C$ is a map from $L$ to $L$. Let us further assume that $H$ is positive semidefinite. Let $w\in N(A)\subset K$ be an element of the nullspace of $A$, i.e. let $A w = 0$. For $v=(w,0)^T\in V$ we then have 
\begin{equation}
	v^\star H v=0.
\end{equation}
Since $H$ is positive semidefinite this implies (see \cite{matrixanalysis}) that we already have
\begin{equation}
	Hv = \begin{pmatrix}
		0 \\
		B^\star w
	\end{pmatrix} = 0,
\end{equation}
or 
\begin{equation}\label{eqn.nulla}
	N(A) \subset N(B^\star).
\end{equation}
In a similar fashion we can establish that the nullspace of $C$ is contained in the null space of $B$:
\begin{equation}\label{eqn.nullc}
	N(C) \subset N(B).
\end{equation}
The relations for the nullspaces are equivalent to these relations for the ranges of $A$ and $C$:
\begin{align}
	R(B) & \subset R(A) \\
	R(B^\star) & \subset R(C).
\end{align}
Because of these properties $H$ is said to have the \emph{column inclusion property}\cite{matrixanalysis}. It follows from equations (\ref{eqn.nulla}) and (\ref{eqn.nullc}) that 
\begin{equation}
	N(A) \oplus N(C) \subset N(H).
\end{equation}
This implies that the rank of $H$ is less than or equal to the sum of the ranks of $A$ and $C$:
\begin{align}
	\rank\ H & = n - \dim N(H) \\
	& \le n - (\dim N(A) + \dim N(C)) \\
	& = k-\dim N(A) + (n-k) - \dim N(C) \\
	& = \rank\ A + \rank\ C
\end{align}
We have equality if and only if $N(H) = N(A)\oplus N(C)$. In the following, matrices $H$ for which this equality holds will be of particular interest to us which is why we make the following definition:

\begin{definition}
	Let $H$ be a positive semidefinite Hermitian matrix that is partitioned as in equation (\ref{eqn.defh}). We say that $H$ is of \emph{maximal rank} if and only if $N(H) = N(A)\oplus N(C)$\footnote{We should be precise and say that $H$ is of maximal rank with respect to the partition in equation (\ref{eqn.defh}). For the sake of readability we will refrain from doing so and rely on the reader to infer the partition from the context.}. 
\end{definition}

When $H$ is of maximal rank, it vanishes on
\begin{equation}
	N(A)\oplus N(C),
\end{equation}
and is positive definite when restricted to the sum of the ranges of $A$ and $C$:
\begin{equation}
	R(A) \oplus R(C).
\end{equation} 
In section \ref{sec.ext} we will use this property to extend results that are valid for positive definite matrices to partitioned matrices of maximal rank. We need one more notion before we can formulate these results.

\section{The generalized determinant}\label{sec.gendet}
A Hermitian matrix $H$ defines a nonsingular map from its range $R(H)$ to its range $R(H)$. If the nullspace $N(H)$ is nonzero, i.e. if $H$ is singular, the determinant of $H$ vanishes. The determinant thus contains no information about the nonsingular map that is $H$ restricted to $R(H)$. To recover this information we introduce a generalized determinant $\detp$:

\begin{definition}\label{def.gendet}
	Let $H$ be Hermitian. Let $\bar H$ be $H$ restricted to the range of $H$:
	\begin{equation}
		\bar H = H \vert_{R(H)} : R(H) \longrightarrow R(H) \\
	\end{equation}
	For $H\ne 0$ we then set 
	\begin{equation}
		\detp H = \det \bar H.
	\end{equation}
	For $H=0$ we set $\detp H = 1$.
\end{definition}

We note a number of properties of the generalized determinant:

\begin{lemma}\label{lem.gentdetprop}
	Let $H$ be Hermitian of rank $r\le n$. Let $\lambda_i, i=1, \ldots, n$, be the eigenvalues of $H$. Let us assume that they are ordered in such a way that $\lambda_i = 0$, for $i>r$. We then have:
	\begin{enumerate}
		\item If $H$ is of full rank (i.e. if $r=n$) we have
		\begin{equation}
			\detp H = \det H = \prod_i \lambda_i
		\end{equation}
		\item For $r<n$ we have
		\begin{equation}
			\detp H = \prod_{i=1}^r \lambda_i
		\end{equation}
		\item We have
		\begin{equation}\label{eqn.limit}
			\detp H = \lim_{\epsilon \rightarrow 0} \frac{\det( H + \epsilon I )}{\epsilon^{n-r}}
		\end{equation}
		\item For $c>0$ we have
		\begin{equation}
			\detp c H = c^{r} \detp H
		\end{equation}
	\end{enumerate}   
\end{lemma}

\begin{proof}
	All of these identities follow from the fact that $H = U D U^\star$ for a unitary $U$ and a diagonal $D$ that contains the eigenvalues of $H$ on the diagonal.
\end{proof}

We note that the generalized determinant lacks many of the properties the usual determinant has. In particular, it is not a continuous function of $H$, and the generalized determinant of the product of two matrices is not the product of the two generalized determinants. 

\section{The extensions}\label{sec.ext}
With the preparations in section \ref{sec.partition}, and the definition of the generalized determinant in the last section, we now want to extent three results to singular matrices: Fischer's inequality, the determinant equality for the Schur complement, and Banachiewicz's form of the inverse of a matrix.

\subsection{Fischer's inequality}
Fischer's inequality states that for a Hermitian positive semidefinite $H$ as in equation (\ref{eqn.defh}) we have (see e.g. \cite{horn}):
\begin{equation}
	\det H \le \det A\; \det C.
\end{equation}
Furthermore, if $H$ is positive definite, we have equality if and only if $B=0$. If $H$ is singular, the determinant of $H$ vanishes and the inequality is no longer much of a constraint. We can also no longer infer that $B$ vanishes in the case of equality. 

We can do better when $H$ is of maximal rank. In this case $H$ vanishes on $N(A)\oplus N(C)$ and is positive definite on 
\begin{equation}
	R(A)\oplus R(C).
\end{equation}
We now look at the restriction of $H$ and all its submatrices to this space and use the Fischer's inequality there. Note that because $H$ is positive semidefinite it has the column (and row) inclusion property and we have
\begin{equation}
	N(A) \subset N(B^\star) 
\end{equation}
and
\begin{equation}
	R(B^\star) \subset R(C). 
\end{equation}
$B^\star$ thus vanishes on the complement of $R(A)$ and maps into $R(C)$. The restriction of $B^\star$ to $R(A)$ thus keeps the nontrivial part of $B^\star$. The same is true for $B$ and its restriction to $R(C)$. As in definition \ref{def.gendet}, we denote the restriction of $H$ to its range by $\bar H$. We obtain
\begin{align}
	\detp H & = \det \bar H \\
	&\le \det \bar A\; \det \bar C \\
	& = \detp A\; \detp C.
\end{align}
Because we are looking at the restriction of $H$ to $R(A)\oplus R(C)$ and because $H$ is positive definite on $R(A)\oplus R(C)$, we also get that $B$, when restricted to $R(C)$, vanishes if and only if equality holds above. Since $B$ vanishes on $N(C)$, this is the case if and only if
\begin{equation}
	B = 0.
\end{equation}
We thus have the following result:

\begin{proposition}\label{prop.fischer}
	Let a Hermitian $H$ be positive semidefinite and partitioned as in equation (\ref{eqn.defh}). If $H$ is of maximal rank then
	\begin{equation}
		\detp H \le \detp A\; \detp C,
	\end{equation} 
	with equality if and only if $B = 0$. 
\end{proposition}

Let us note that we arrived at this result in two steps. Because $H$ is positive semidefinite we know that $B$ restricted to $N(C)$ is zero. That the restriction of $B$ to the range $R(C)$ is also zero follows from Fischer's equality for the positive definite matrix that is $H$ restricted to $R(A)\oplus R(C)$.

We note that to show the inequality we could have also used equation (\ref{eqn.limit}) of lemma \ref{lem.gentdetprop}.

\subsection{Schur complement}
The Schur complement arises naturally when using Gaussian elimination to solve a linear equation. Let $H$ be as in equation (\ref{eqn.defh}) and let us assume we want to solve
\begin{align}
0 & = H \begin{pmatrix}
		k \\
		l
	\end{pmatrix} \\
	& = \begin{pmatrix}
		A k + B l \\
		B^\star k + C l
 \end{pmatrix}.	\label{eqn:lineq}
\end{align}
Assuming that $A$ is invertible we can solve for $k$ in the first equation of (\ref{eqn:lineq}) to obtain
\begin{equation}\label{eqn.gauss1}
	k = - A^{-1}B l.
\end{equation}
The second equation of (\ref{eqn:lineq}) then gives
\begin{equation}\label{eqn.gauss2}
	(C - B^\star A^{-1}B) l = 0.
\end{equation}
The expression in the parentheses is called the Schur complement of $A$ in $H$ and is denoted by $H/A$:
\begin{equation}
	H/A = C - B^\star A^{-1}B
\end{equation}
For a positive semidefinite matrix $H$ we may generalize the definition to
\begin{equation}
	H/A = C - B^\star A^+B,
\end{equation}
where $A^+$ is the Moore-Penrose inverse of $A$ (again, see \cite{roger} and \cite{matrixanalysis}). Note that the choice of generalized inverse does not matter here since $R(B)\subset R(A)$. Emilie Haynsworth introduced the name and showed that the Schur complement possesses many interesting properties \cite{hayns} (see \cite{horn} for a detailed exposition). In \cite{schur} Issai Schur showed that for a positive definite $H$ the determinant of $H$ satisfies
\begin{equation}
	\det H = \det A\; \det H/A.
\end{equation}
We now want to show that this equality also holds for the generalized determinant. Again, we focus our attention on $R(A)\oplus R(C)$ where $H$ is positive definite. Using the notation from the previous section we obtain:
\begin{align}
	\detp H & = \det \bar H \\
	& = \det \bar A \; \det \bar H / \bar A \\
	& = \detp A  \; \det \bar H / \bar A
\end{align}
We need to convince ourselves that the last determinant is equal to $\detp H/A$. To check this, we need to show that 
\begin{align}
	N(H/A) & = N( C - B^\star A^+ B ) \\ 
	& = N(C).
\end{align}
We already know that $N(C) \subset N(H/A)$. To show equality let $l\in N(H/A)$. This $l$ thus satisfies equation (\ref{eqn.gauss2}) that was obtained through Gaussian elimination. We define $k$ according to equation (\ref{eqn.gauss1}): 
\begin{equation}
	k = - A^+ B l.
\end{equation}
It then follows that 
\begin{equation}
	H\begin{pmatrix}
		k\\
		l
	\end{pmatrix} = 0.
\end{equation}
Since $H$ is of maximal rank this implies that $(k,l)^T\in N(A)\oplus N(C)$. In particular, we have $l\in N(C)$. We thus obtain our second result:

\begin{proposition}\label{prop.schur}
 Let a Hermitian $H$ be positive semidefinite and partitioned as in equation (\ref{eqn.defh}). If $H$ is of maximal rank then
	\begin{equation}
		\detp H = \detp A\; \detp H/A.
	\end{equation}	
\end{proposition}

\subsection{The inverse}
The last result concerns the inverse of $H$. When $H$ is positive definite its inverse can be written as
\begin{equation}\label{eqn.bana}
	H^{-1} = \begin{pmatrix}
		A^{-1} + A^{-1}B(H/A)^{-1}B^\star A^{-1} & - A^{-1} B (H/A)^{-1} \\
		- (H/A)^{-1} B^\star A^{-1} & (H/A)^{-1}
	\end{pmatrix}.
\end{equation} 
This is a remarkable formula because to find the inverse of $H$ we just need to invert $A$ and $H/A$. This formula was first established by Banachiewicz \cite{bana} (see also \cite[p.112]{frazer}). We now want to adapt this formula to our situation where $H$ might be singular but is of maximal rank. 

Again, we start by looking at the restriction $\bar H$ of $H$ to its range $R(A)\oplus R(C)$. $\bar H$ is positive definite and we can express its inverse in the form of equation (\ref{eqn.bana}). In the last section we established that the ranges and nullspaces of $C$ and $H/A$ coincide so that by replacing the inverses in equation (\ref{eqn.bana}) with Moore-Penrose inverses we obtain the unique Moore-Penrose inverse of $H$.

\begin{proposition}\label{prop.bana}
	Let a Hermitian $H$ be positive semidefinite and partitioned as in equation (\ref{eqn.defh}). If $H$ is of maximal rank then its Moore-Penrose inverse is given by
	\begin{equation}\label{eqn.banapenrose}
	H^{+} = \begin{pmatrix}
		A^{+} + A^{+}B(H/A)^{+}B^\star A^{+} & - A^{+} B (H/A)^{+} \\
		- (H/A)^{+} B^\star A^{+} & (H/A)^{+}
	\end{pmatrix}.
	\end{equation}
\end{proposition}

This result can also be deduced from \cite[Theorem 4.6]{ouellette}. 

\section{Matrix completion}
Let $H$ be a Hermitian matrix in $M_n(\C)$ that is partitioned as follows:
\begin{equation}\label{eqn.defnewh}
	H = \begin{pmatrix}
		A & B & X \\
		B^\star & C & D \\
		X^\star & D^\star & E
	\end{pmatrix}
\end{equation}
As in the introduction we use the notation from \cite{matrixanalysis} to denote submatrices of $H$. For $\alpha\subset \{1, \ldots, n\}$ let
\begin{equation}
	H[\alpha]
\end{equation}
be the submatrix of $H$ with row and column indices in $\alpha$. For $\alpha,\beta\subset \{1, \ldots, n\}$ we let
\begin{equation}
	H[\alpha, \beta] 
\end{equation}
be the submatrix with row indices in $\alpha$ and column indices in $\beta$. Now let $\alpha$ and $\beta$ be such that
\begin{equation}
	H[\alpha] = \begin{pmatrix}
		A & B \\
		B^\star & C
	\end{pmatrix}
\end{equation}
and 
\begin{equation}
	H[\beta] = \begin{pmatrix}
		C & D \\
		D^\star & E
	\end{pmatrix}.
\end{equation}
For $\gamma = \alpha\cap\beta$ we then have
\begin{equation}
	H[\gamma] = C,
\end{equation}
and for the matrix $X$ in the upper right corner we have
\begin{equation}
	H[ \alpha-\gamma, \beta-\gamma ] = X.	
\end{equation}
We now want to know under what conditions we can choose $X$ so that the matrix $H$ is positive semidefinite. Before we state the result we note this helpful theorem:

\begin{theorem}\label{theo.albert}
	Let $H$ be Hermitian and partitioned as in equation (\ref{eqn.defh}). Then these two statements are equivalent:
	\begin{enumerate}
		\item $H$ is positive semidefinite.
		\item $A$ and $H/A$ are positive semidefinite, and $R(B)\subset R(A)$.
	\end{enumerate}	
\end{theorem}

A proof of this result can be found in \cite{albert}. We now have:

\begin{theorem}\label{theo.x}
	Let $H\in M_n(\C)$ be partitioned as in (\ref{eqn.defnewh}) and let $H[\alpha]$ and $H[\beta]$ be positive semidefinite. Setting
	\begin{align}
		X & = B C^+ D \\
		& = H[\alpha-\gamma, \gamma] H[\gamma]^+ H[\gamma, \beta-\gamma]
	\end{align}
	turns $H$ into a positive semidefinite matrix.
\end{theorem}

\begin{proof}
	We apply theorem \ref{theo.albert} to the positive semidefinite matrices $H[\alpha]$ and $H[\beta]$ to obtain:
	\begin{align}
		C & \ge 0 \\
		H[\alpha]/C & \ge 0 \label{eqn.hac}\\
		H[\beta]/C  & \ge 0 \label{eqn.hbc} \\
		R(B^\star) & \subset R(C) \label{eqn.bsc} \\				
		R(D) & \subset R(C) \label{eqn.dsc}
	\end{align}	
	The Schur complement $H/C$ is given by
	\begin{equation}
		H/C = \begin{pmatrix}
			A - B C^+ B^\star & X - B C^+ D \\
			X^\star - D^\star C^+ B^\star & E - D^\star C^+ D
		\end{pmatrix}.
	\end{equation}
	If we choose 
	\begin{equation}
		X = B C^+ D,
	\end{equation}
	and also recognize the expressions for $H[\alpha]/C$ and $H[\beta]/C$ we find
	\begin{equation}
		H/C = \begin{pmatrix}
			H[\alpha]/C & 0 \\
			0 & H[\beta]/C
		\end{pmatrix}.
	\end{equation}
	Because of equations (\ref{eqn.hac}) and (\ref{eqn.hbc}) we have
	\begin{equation}
		H/C \ge 0.
	\end{equation}
	Equations (\ref{eqn.bsc}) and (\ref{eqn.dsc}) then ensure that
	\begin{equation}
		R(B^\star\; D ) \subset R(C).
	\end{equation}
	Since we also have $C\ge 0$ we can use the theorem \ref{theo.albert} one more time, this time in the other direction, to obtain
	\begin{equation}
		H \ge 0.
	\end{equation} 
	This completes the proof.
\end{proof}

This result appeared already in \cite{smith}. We have provided a different proof that makes use of theorem \ref{theo.albert}. 

It turns out that the choice of $X$ in theorem \ref{theo.x} gives $H$ unique properties. It is in these characterizations of $X$ that we go beyond the results in \cite{smith} because we include the case in which $H$ is singular. The first result characterizes $X$ as the unique extension of $H$ that maximizes the determinant of $H$. If $H$ is nonsingular we can just talk about the regular determinant of $H$. If $H$ is singular we have to use the generalized determinant that we introduced in section \ref{sec.gendet}.

\begin{theorem}
	Let $H\in M_n(\C)$ be partitioned as in (\ref{eqn.defnewh}) and let $H[\alpha]$ and $H[\beta]$ be positive semidefinite and of maximal rank. The choice 
	\begin{equation}
		X = B C^+ D
	\end{equation}
	is the unique choice for $X$ for which $H$ is positive semidefinite, of maximal rank, and for which the (generalized) determinant is maximal.  
\end{theorem}

\begin{proof}
We have shown in the last theorem that $H$ is positive semidefinite if we set $X = B C^+ D$. $H$ is also of maximal rank. In general we have
\begin{equation}
	\rank\ H \le \rank\ A + \rank\ C + \rank\ E.
\end{equation}
For a positive semidefinite matrix rank is additive over the Schur complement (see \cite{ouellette}) so that we actually have equality:
\begin{align}
	\rank\ H & = \rank\ C + \rank\ H/C \\
	& = \rank\ C + \rank\ H[\alpha]/C + \rank\ H[\beta]/C \\
	& = \rank\ C + \rank\ A + \rank\ E
\end{align}
To establish the last equality we have used the assumption that both $H[\alpha]$ and $H[\beta]$ are of maximal rank. Thus $X=B C^+ D$ turns $H$ into a matrix of maximal rank. We now want to show that it is the only such choice that also maximizes the determinant.

Let us now assume that $H$ is positive semidefinite and of maximal rank so that we can make use of propositions \ref{prop.fischer} and \ref{prop.schur}. Because of proposition \ref{prop.schur} we have
\begin{equation}
	\detp H = \detp C \; \detp H/C.
\end{equation}
Because of proposition \ref{prop.fischer} we have
\begin{equation}
	\detp H/C \le \detp H[\alpha]/C \detp H[\beta]/C,
\end{equation}
with equality if and only if
\begin{equation}
	X = B C^+ D.
\end{equation}
It follows that this $X$ is the unique choice that maximizes the determinant of $H$.
\end{proof}

For a nonsingular matrix $H$ we can use the determinant to find the inverse of $H$ (see \cite{matrixanalysis}):
\begin{equation}\label{eqn.inverseformula}
	H^{-1} = \frac{1}{\det H}\left(\frac{\partial}{\partial h_{ij}}\det H \right)^T
\end{equation}

Since $X$ was chosen such that the determinant is maximal the derivative in equation (\ref{eqn.inverseformula}) vanishes for indices $i$ and $j$ that denote elements of $X$ itself. It follows that $H^{-1}$ has zeroes in those places where the matrix $X$ sits in $H$. It turns out that this uniquely determines $X$ even if $H$ is only positive semidefinite and we have to talk about the Moore-Penrose inverse of $H$ instead.

\begin{theorem}
	Let $H\in M_n(\C)$ be partitioned as in (\ref{eqn.defnewh}) and let $H[\alpha]$ and $H[\beta]$ be positive semidefinite and of maximal rank. The choice 
	\begin{equation}
		X = B C^+ D
	\end{equation}
	is the unique choice for $X$ for which $H$ is of maximal rank and for which $H^{+}$ has zeroes in those places where $X$ sits in $H$.
\end{theorem}

\begin{proof}
	Because $H$ is of maximal rank we can use proposition \ref{prop.bana} to express the Moore-Penrose inverse of $H$ in terms of $C^+$ and $(H/C)^+$. If $H^+$ is to have zeroes where $X$ is in $H$ then we must have
	\begin{equation}
		(H/C)^+ = \begin{pmatrix}
			(H[\alpha]/C)^+ & 0 \\
			0 & (H[\beta]/C)^+  
		\end{pmatrix},
	\end{equation}
	which can only be the case if $X=B C^+ D$.
\end{proof}

For completeness we give the Moore-Penrose inverse for $H$:

\begin{equation}
	H^+ = \begin{pmatrix}
		(H[\alpha]/C)^+ & -(H[\alpha]/C)^+ B C^+ & 0 \\
		- C^+ B^\star (H[\alpha]/C)^+ & \Xi & - C^+ D (H[\beta]/C)^+ \\
		0 & - (H[\beta]/C)^+ D^\star C^+ & (H[\beta]/C)^+ 
	\end{pmatrix},
\end{equation}
with
\begin{equation}
	\Xi = C^+ + C^+ B^\star (H[\alpha]/C)^+ B C^+ + C^+ D (H[\beta]/C)^+ D^\star C^+.
\end{equation}

\section{Conclusion}
The problem of how to complete partial Hermitian matrices arises frequently in practical applications (see \cite{otherpaper} for an example from finance). This problem was solved in \cite{grone} for partial matrices whose corresponding graph is chordal. The procedure provided in \cite{grone} was improved upon in \cite{smith} by giving a way to calculate whole blocks of the completion at once. For positive definite matrices the resulting completion is singled out by two uniqueness results. It is the unique matrix that maximizes the determinant, and it is the unique matrix whose inverse has zeroes in those places that were unspecified in the original matrix. In this paper we have extended these uniqueness results to include semidefinite matrices. To make this extension possible we needed to introduce a generalized determinant that gives the determinant of the nontrivial part of a Hermitian matrix. We also needed to focus on matrices whose rank is determined solely by the rank of its diagonal matrices. For these matrices the same uniqueness results hold that hold for positive definite matrices. 

\section*{Acknowledgement}
I would like to thank Patrick B\"uchel for his support during the creation of this work as well as Horst K\"ohler and Thomas Streuer for their initial push to look into maximal determinant completions of matrices, and for their support during the creation of this work.

\end{document}